\documentclass[11pt,a4paper,openbib]{article}
\usepackage[ansinew]{inputenc}
\usepackage{amssymb}
\usepackage{amsmath}
\usepackage{amsthm} 
\usepackage{amsxtra} 
\usepackage{graphicx}
{\tiny }
\newcommand {\rel} {{\mathbb R}}

\newcommand {\nat} {{\mathbb N}}

\newcommand {\ganz} {{\mathbb Z}}

\newcommand {\sphere} {{\mathbb S}}

\newcommand {\Will} {{\mathcal{W} }}
\newcommand {\Wil} {{\mathcal{E} }}

\begin{document}
	
\newtheorem{theorem}{Theorem}[section]
\newtheorem{definition}{Definition}[section]
\newtheorem{proposition}{Proposition}[section]
\newtheorem{lemma}{Lemma}[section]
\newtheorem{corollary}{Corollary}[section]
\newtheorem{remark}{Remark}[section]
\newtheorem{example}{Example}
	
\title{Corrections to: 
``The Willmore flow of Hopf-tori in the $3$-sphere''}
	
	\author{Ruben Jakob}

	\maketitle

\begin{abstract}
This erratum addresses a logical mistake in the author's 
article [Jakob, R. The Willmore flow of Hopf-tori in the $3$-sphere. Journal of Evolution Equations {\bf 23}, No. 72 (2023)] which resulted in two wrong assertions in parts (II) and (III) of Theorem 1 in the author's cited paper  
and in an inaccuracy in the formulation of the second part of Proposition 6 in that paper. We will not only point out these mistakes and their corrections, but we will additionally give a concrete counterexample to Statement (III) in Theorem 1 of 
the author's cited article which will automatically 
imply the optimality of the corrected version of that 
statement, as it is formulated in this erratum. 
\end{abstract}

\section{Corrections of mistakes in \cite{Ruben.Hopf.Willmore.2023}}
\label{Corrections}
In the author's article \cite{Ruben.Hopf.Willmore.2023}  
the author investigated the long-term behaviour of the classical Willmore flow 
\begin{equation}  \label{Willmore.flow}
\partial_t F_t = -\frac{1}{2}  \,
\Big{(} \triangle_{F_t}^{\perp} \vec H_{F_t} + Q(A^{0}_{F_t})
(\vec H_{F_t}) \Big{)} 
\end{equation}
i.e. the $L^2$-gradient-flow of the Willmore energy 
\begin{equation} \label{Willmore.functional}
\Will(F):= \int_{\Sigma} 1 + \frac{1}{4} \, 
\mid \vec H_F \mid^2 \, d\mu_F,
\end{equation}
being restricted to families of smooth immersions
$F_t$ of some compact smooth torus $\Sigma$ into the standard 
$3$-sphere ``$\sphere^3$''.\\ 
The main result, Theorem 1, in \cite{Ruben.Hopf.Willmore.2023} addresses the long-term behaviour and full convergence of flow lines of flow (\ref{Willmore.flow}) and claims in its 
third part that every flow line 
$\{F_t\}$ of flow (\ref{Willmore.flow}) which starts 
moving in a - so-called - \emph{simple}
\footnote{See here Definition 3 in \cite{Ruben.Hopf.Willmore.2023}.}
parametrization $F_0:\Sigma \longrightarrow \sphere^3$ 
of a Hopf-torus in $\sphere^3$ with Willmore energy 
$\Will(F_0) \leq \frac{8\pi^2}{\sqrt{2}}$ 
consists of smooth embeddings of $\Sigma$ into $\sphere^3$ for sufficiently large $t$ and converges - up to time dependent, smooth reparametrization - fully and smoothly 
to a smooth embedding of the Clifford torus in $\sphere^3$ 
- at least up to some conformal transformation of $\sphere^3$. However, {\bf the claimed threshold $\frac{8\pi^2}{\sqrt{2}}$ has turned out to be too large} and has to be decreased exactly to the number $4\pi^2$. \\
{\bf Secondly}, in the second statement of Theorem 1 in \cite{Ruben.Hopf.Willmore.2023} addressing subconvergence 
of any flow line of flow (\ref{Willmore.flow})
which starts moving in a \emph{simple} parametrization of 
some Hopf-torus in $\sphere^3$ - without any restriction 
on its Willmore energy - the smooth limit immersion  
$\hat F$ in formula (14) of \cite{Ruben.Hopf.Willmore.2023} 
{\bf does in general not have to be simple} anymore. \\
{\bf Thirdly}, in the second part of Proposition 6 of \cite{Ruben.Hopf.Willmore.2023} one has to replace the 
interval $\big{(}2\pi, \frac{8\pi}{\sqrt{2}}\big{]}$
by the disjoint union of intervals $(2\pi, 4\pi) \cup  
\big{(}4\pi,\frac{8\pi}{\sqrt{2}}\big{]}$, i.e.    
there indeed is no critical value of the elastic energy 
$\Wil:C^{\infty}_{\textnormal{reg}}(\sphere^1,\sphere^2)
\longrightarrow \rel$
\footnote{As in Proposition 3 of 
\cite{Ruben.Hopf.Willmore.2023} we shall define the elastic energy $\Wil$ of some smooth, regular, closed path 
$\gamma:\sphere^1 \longrightarrow \sphere^2$ by 
$\Wil(\gamma):=
\int_{\sphere^1} 1+|\vec{\kappa}_{\gamma}|^2 \, \,d\mu_{\gamma}$.}   
in the set of real numbers $(2\pi, 4\pi) \cup \big{(}4\pi,\frac{8\pi}{\sqrt{2}}\big{]}$, 
whereas {\bf the number $4\pi$ is actually the elastic energy 
of an elastic curve}, namely of the elastica 
which consists of two consecutive loops through an 
arbitrarily chosen great circle in $\sphere^2$.  \\ 
These three mistakes appeared in highlighted statements of 
\cite{Ruben.Hopf.Willmore.2023}, and they were caused by only one subtle logical mistake, namely to apply the terminology of Definition 3 in \cite{Ruben.Hopf.Willmore.2023} to 
general smooth {\bf immersions} of $\Sigma$ into $\sphere^3$ - which don't necessarily have to be 
embeddings - within the proofs of Statements (II) and (III) 
of Theorem 1 in \cite{Ruben.Hopf.Willmore.2023}. 
More precisely, in Definition 3 of \cite{Ruben.Hopf.Willmore.2023} the author decided to distinguish smooth immersions 
$F:\Sigma \longrightarrow \sphere^3$ in terms of their induced maps in singular homology, precisely in terms of 
the homomorphisms 
\begin{equation} \label{homological.stuff}
(F_{*})_2: H_2(\Sigma,\ganz) \longrightarrow
H_{2}(F(\Sigma),\ganz),
\end{equation}
in degree $2$. Since $F:\Sigma \longrightarrow F(\Sigma)\subset \sphere^3$ is a continuous map between 
two topological spaces, the homomorphism in (\ref{homological.stuff}) is certainly well-defined, but it {\bf does not necessarily have to be a homomorphism from $\ganz$ to $\ganz$}, in which case  
the image of a generator of $H_2(\Sigma,\ganz)$ under $(F_{*})_2$ in $H_{2}(F(\Sigma),\ganz)$ would topologically characterize the immersion $F$, similarly 
to the classical notion of ``mapping degree'' 
for continuous maps from $\sphere^n$ to $\sphere^n$.  
However, this idea only works out for embeddings of 
$\Sigma$ into $\sphere^3$, not for immersions in general, 
because for a general immersion 
$F:\Sigma \longrightarrow \sphere^3$ the Abelian group  
$H_{2}(F(\Sigma),\ganz)$ might be far more complicated than $\ganz$. In conclusion, the original introduction of ``simple immersions'' in terms of the algebraic criterion in Definition 3 of \cite{Ruben.Hopf.Willmore.2023} is not applicable throughout the proof of Theorem 1 in \cite{Ruben.Hopf.Willmore.2023} and therefore useless - 
actually even misleading.  
Instead, one should have rather introduced 
in \cite{Ruben.Hopf.Willmore.2023} the following more expedient terminology, which we quote from the author's paper \cite{Ruben.MIWF.2025}; 
see Definition 4.3 in \cite{Ruben.MIWF.2025}.           
\begin{definition} \label{Simple.map}
Let $\Sigma$ be a compact smooth torus and 
$F:\Sigma \longrightarrow \sphere^3$ a smooth immersion. 
We call $F$ a \emph{simple parametrization} of the 
immersed torus $F(\Sigma) \subset \sphere^3$, if there holds $\sharp \{F^{-1}(z)\} =1$ in $\mathcal{H}^2$-a.e. 
$z \in F(\Sigma)$. 
\end{definition}
One should compare Definition \ref{Simple.map} 
with Remark 4.2 in \cite{Ruben.MIWF.2025}, where the author points out why any smooth immersion 
$F:\Sigma \longrightarrow \sphere^3$ 
satisfying $\Will(F)< 4 \pi^2$ must be a 
\emph{simple parametrization} of $F(\Sigma)$ 
in the sense of Definition \ref{Simple.map} 
above, and secondly that in the special case in which $F(\Sigma)$ is actually embedded into $\sphere^3$, i.e. 
a smooth compact torus in $\sphere^3$ in the sense of differential topology, the condition in Definition \ref{Simple.map} above is satisfied, if and only if 
the induced homomorphism in 
equation (\ref{homological.stuff}) is isomorphic, 
i.e. if and only if a generator of $H_2(\Sigma,\ganz)\cong \ganz$ is taken by $(F_{*})_2$ to a generator of   
$H_{2}(F(\Sigma),\ganz)\cong \ganz$. 
This latter fact disguised the basic logical mistake 
in \cite{Ruben.Hopf.Willmore.2023} which was pointed out above, below equation (\ref{homological.stuff}). \\
Now, since the condition ``$\Will(F)< 4 \pi^2$'' forces $F$ to be a \emph{simple parametrization} of the immersed torus $F(\Sigma)\subset \sphere^3$ in the sense of our revised Definition \ref{Simple.map} and since the Willmore energy along any flow line of the Willmore flow (\ref{Willmore.flow}) cannot increase, any flow line 
of this flow in $\sphere^3$ starting in a Hopf-torus with Willmore energy below $4 \pi^2$ can - indeed - only smoothly subconverge to a \emph{simple} Willmore-Hopf-torus in 
$\sphere^3$, as the author - incorrectly - claimed 
in Statement (II) of Theorem 1 in \cite{Ruben.Hopf.Willmore.2023} without any energy 
condition on $F_0$. 
One should therefore feel confident that the corrected threshold $4 \pi^2$ in Statement (III) of Theorem 1 in 
\cite{Ruben.Hopf.Willmore.2023} will be sufficiently small such that Theorem \ref{corrected.StatementIII.Theorem.1}
below can be proved along the lines of the proof 
on pp. 24--30 in \cite{Ruben.Hopf.Willmore.2023}, 
at least without any significant changes. 
Now, on a more technical level and according to the beginning of the proof of Statement (III) of Theorem 1 in \cite{Ruben.Hopf.Willmore.2023} 
one verifies first of all that any path $\gamma$ in 
$C^{\infty}_{\textnormal{reg}}(\sphere^1,\sphere^2)$ 
with elastic energy $\Wil(\gamma)<4\pi$ 
can only parametrize a simple loop through its trace. 
Hence, the proof of Proposition 6 in \cite{Ruben.Hopf.Willmore.2023} shows - without any alteration - that there \emph{cannot be any critical value} of $\Wil$ in the interval $(2\pi,4\pi)$, and therefore 
the entire beginning of the proof of Statement (III) of Theorem 1 on pp. 24--25 in \cite{Ruben.Hopf.Willmore.2023} works out without any further complication, provided 
one requires ``$\Will(F_0)< 4 \pi^2$'' 
for the initial immersion $F_0$ - and not only 
``$\Will(F_0)\leq \frac{8 \pi}{\sqrt 2}$'', as 
incorrectly claimed in Statement (III) of Theorem 1 in 
\cite{Ruben.Hopf.Willmore.2023} - and then systematically 
replaces $\frac{8 \pi}{\sqrt 2}$ by $4\pi$ 
on pp. 24--25 in \cite{Ruben.Hopf.Willmore.2023}.      
Moreover, formula (82) on p. 28 in \cite{Ruben.Hopf.Willmore.2023} becomes redundant, 
which actually simplifies the argument
at that point of Step 3 of the entire proof,
because the constructed parametrizations 
$Y_t$ of the Hopf-tori 
$\pi^{-1}(\textnormal{trace}(P(t,0,\gamma_0)))$ are  
indeed simple in the sense of Definition \ref{Simple.map} above because of formula (81) and the second part of Remark 4 in \cite{Ruben.Hopf.Willmore.2023}. One can therefore easily conclude that 
\begin{equation} \label{nice.estimate} 
\Will(Y_t)= \Will(\pi^{-1}(\textnormal{trace}(P(t,0,\gamma_0))))
<8 \pi, \,\,\, \textnormal{for every}\,\, t\geq t_0, 
\end{equation}
for some sufficiently large time $t_0$, on account of Propositions 4 and 6 in \cite{Ruben.Hopf.Willmore.2023} 
and on account of the strict monotonicity of the Willmore energy along non-constant flow lines of the Willmore flow.
Hence, the Li-Yau inequality implies immediately that     
$Y_t:\Sigma \stackrel{\cong}\longrightarrow 
\pi^{-1}(\textnormal{trace}(P(t,0,\gamma_0)))$ 
is a smooth diffeomorphism, for every $t\geq t_0$, 
just as desired.
Finally, below formula (85) on p. 29 in \cite{Ruben.Hopf.Willmore.2023} one can slightly simplify the entire argument in the following way: 
By means of inequality (\ref{nice.estimate}) above, 
the $C^m$-convergence of the immersions $Y_t$ 
to some smooth immersion $F^*$ in formula (85) on p. 29 in \cite{Ruben.Hopf.Willmore.2023} and the lower semicontinuity 
of the Willmore functional w.r.t. $C^m$-convergence, one 
obtains for this limit immersion $F^*$ that
$\Will(F^*)\leq \liminf_{t\to \infty} \Will(Y_t) <8\pi$, hence - again on account of the Li-Yau inequality - that the limit immersion $F^*$ in formula (85) of \cite{Ruben.Hopf.Willmore.2023} must be a smooth diffeomorphism from $\Sigma$ onto its image, which indeed has to be the Clifford-torus in $\sphere^3$, 
as explained below formula (85) in \cite{Ruben.Hopf.Willmore.2023}. 
Hence, up to these minor modifications one can prove as in \cite{Ruben.Hopf.Willmore.2023} the following 
convergence result for the Willmore flow 
of Hopf-tori in $\sphere^3$. 
\begin{theorem}[Corrected Statement (III) in 
Theorem 1 of \cite{Ruben.Hopf.Willmore.2023}]
\label{corrected.StatementIII.Theorem.1} 	   
Let $F_0:\Sigma \longrightarrow \sphere^3$ be a smooth immersion of a smooth compact torus $\Sigma$ into  $\sphere^3$ which parametrizes a Hopf-torus with 
Willmore energy $\Will(F_0) < 4\pi^2$. Then 
there is a smooth family of smooth diffeomorphisms
$\Psi_t:\Sigma \longrightarrow \Sigma$ such that the reparametrization $\{\mathcal{P}(t,0,F_0) \circ \Psi_t\}_{t\geq 0}$ of the flow line $\{\mathcal{P}(t,0,F_0)\}_{t\geq 0}$ of the Willmore flow (\ref{Willmore.flow}) in $\sphere^3$, starting in $F_0$ at time $t=0$, consists of smooth diffeomorphisms between $\Sigma$ and their images in $\sphere^3$ for sufficiently large $t$, and $\{\mathcal{P}(t,0,F_0) \circ \Psi_t\}$ converges fully in $C^m(\Sigma,\rel^4)$, for each $m\in \nat_0$, to a smooth diffeomorphism between $\Sigma$ and the Clifford torus in $\sphere^3$ - up to some appropriate isoclinic rotation of $\sphere^3$.   
\end{theorem} 
\noindent
Now, taking here again Proposition 6 in \cite{Ruben.Hopf.Willmore.2023} and also Theorem \ref{bad.convergence} below into account, we can obtain 
the complete picture regarding 
\emph{convergence of the Willmore flow of Hopf-tori 
starting in a simple immersion - as introduced above in Definition \ref{Simple.map} - with Willmore energy below the originally claimed threshold $\frac{8\pi^2}{\sqrt{2}}$} 
by means of the following dichotomy.
\begin{theorem}
[Alternative revised version of Statement (III) in 
Theorem 1 of \cite{Ruben.Hopf.Willmore.2023}] 
\label{Dichotomy} 	
Let $F_0:\Sigma \longrightarrow \sphere^3$ be a smooth and simple immersion of a smooth compact torus $\Sigma$ into $\sphere^3$ which parametrizes a Hopf-torus 
with Willmore energy $\Will(F_0) \leq \frac{8\pi^2}{\sqrt{2}}$. Then one and only one of the 
following two alternatives must hold, and the second 
alternative cannot be ruled out in general.
\begin{itemize} 
\item[1)] 
There is a smooth family of smooth diffeomorphisms
$\Psi_t:\Sigma \longrightarrow \Sigma$ such that the reparametrization $\{\mathcal{P}(t,0,F_0) \circ \Psi_t\}_{t\geq 0}$ of the flow line $\{\mathcal{P}(t,0,F_0)\}_{t\geq 0}$ of the Willmore flow (\ref{Willmore.flow}) in $\sphere^3$, starting in $F_0$ 
at time $t=0$, converges fully in $C^m(\Sigma,\rel^4)$, for each $m\in \nat_0$, to a smooth embedding of the Clifford torus in $\sphere^3$ - up to some appropriate isoclinic rotation of $\sphere^3$. 
\item[2)] The Willmore energy of the flow line $\{\mathcal{P}(t,0,F_0)\}_{t\geq 0}$ of the Willmore flow (\ref{Willmore.flow}) in $\sphere^3$, starting 
in $F_0$ at time $t=0$, decreases strictly monotonically 
to the value $4\pi^2$, i.e. to the Willmore energy of 
an immersion which covers the Clifford torus exactly \emph{twice}.
\end{itemize}     	
\end{theorem}  

\section{Optimality of the initial energy-threshold 
``$4 \pi^2$'' in Theorem \ref{corrected.StatementIII.Theorem.1}}
\label{Optimality}     
Regarding ``optimality'' we firstly recall Blatt's achievement in \cite{Blatt.2009}, Theorem 5.1, to prove optimality of Kuwert's and Sch\"atzle's famous initial condition ``Willmore energy below $8\pi$'' for flow lines of the classical Willmore flow of immersed $2$-spheres in $\rel^3$ 
to converge smoothly to a round sphere -- see \cite{Kuwert.Schaetzle.Annals}, Theorem 5.2. Similarly, combining Theorems 1.3 and 1.4 
in \cite{Dall.Acqua.Schaetzle.Mueller.2020} - see here 
also Lemma 3.8 in \cite{Dall.Acqua.Schaetzle.Mueller.2020} - the authors of that paper succeeded to prove optimality of the 
initial condition ``Willmore energy below $8\pi$'' 
for flow lines of the classical Willmore flow 
of tori of revolution in $\rel^3$ to 
converge smoothly to a diffeomorphic parametrization 
of the Clifford-torus in $\rel^3$. \\
Now, in Theorem \ref{bad.convergence} below we are going to prove {\bf optimality} of the initial energy-threshold 
``$4 \pi^2$'' appearing in Theorem \ref{corrected.StatementIII.Theorem.1} above. 
However, one should still notice here that our argument 
for optimality is both quantitatively and qualitatively different from the two optimality arguments mentioned above. Not only is the corrected threshold $4 \pi^2$ still much higher than the mentioned $8\pi$-threshold, but the way in which we produce a counterexample to the conclusion of Theorem \ref{corrected.StatementIII.Theorem.1} if we were to slightly 
raise the threshold $4 \pi^2$ to $4 \pi^2+\delta$, 
for any $\delta>0$, is logically different, because  
\emph{we will not and cannot prove any sort of divergence} of some particular flow line of flow (\ref{Willmore.flow}) in $\sphere^3$ starting with Willmore energy in 
$(4 \pi^2, 4 \pi^2+\delta)$. In Theorem \ref{bad.convergence} below we rather prove convergence of the Willmore energy 
along a particular flow line $\{F_t\}_{t \geq 0}$ of the Willmore flow (\ref{Willmore.flow}) in $\sphere^3$ 
- starting in some appropriately chosen \emph{simple} 
immersion $F_0$ with 
$\Will(F_0)\in (4 \pi^2, 4 \pi^2+\delta)$ - 
to the value $4 \pi^2$, which is exactly {\bf twice the Willmore energy of the Clifford-torus}. 
Obviously, this feature of the Willmore flow (\ref{Willmore.flow}) in $\sphere^3$ 
- here restricted to the class of immersed Hopf-tori -  
differs starkly from the above mentioned behaviour of the 
``ordinary'' Willmore flow in $\rel^3$, for which 
the energy level $8\pi$ plays a similar, but more critical r\^ole than the energy level $4 \pi^2$ does for 
the considered flow (\ref{Willmore.flow}) of Hopf-tori in $\sphere^3$.    
Now we prove the following \emph{new statement}.  
\begin{theorem}\label{bad.convergence} 
For every $\delta >0$ there is some  
smooth and simple parametrization $F_0$ of a Hopf-torus in $\sphere^3$ with Willmore energy in $(4\pi^2, 4\pi^2+\delta)$, such that the Willmore energy of the flow line $\{\mathcal{P}(t,0,F_0)\}_{t\geq 0}$ of the Willmore flow (\ref{Willmore.flow}) in $\sphere^3$, starting in $F_0$ at time $t=0$, decreases strictly monotonically 
to the value $4\pi^2$, i.e. to the Willmore energy of 
an immersion which covers the Clifford torus 
in $\sphere^3$ exactly \emph{twice}.   
\end{theorem}  
Before we start to prove this statement we should recall 
here the preparation in Section 4 of \cite{Ruben.Hopf.Willmore.2023} respectively 
in Section 4 of \cite{Ruben.MIWF.2025}, which explains 
why the assertion of Theorem \ref{bad.convergence}  
already follows from the following statement - a statement  
which is much easier to handle than Theorem \ref{bad.convergence} from a technical point of view. 
\begin{theorem}\label{bad.convergence.2} 
For every $\delta >0$ there is some  
smooth, regular and closed path 
$\gamma_0:\sphere^1\longrightarrow \sphere^2$ 
with only one self-intersection and
with elastic energy in $(4\pi, 4\pi+\delta)$, 
such that some appropriate smooth reparametrization 
of the flow line $\{\gamma_t\}_{t\geq 0}$ of the 
elastic energy flow in $\sphere^2$, 
i.e. of the geometric flow 
\begin{eqnarray}  \label{elastic.energy.flow}
\partial_t \tilde \gamma_t =
- \,\Big{(} 2 \, \Big{(}\nabla^{\perp}_{\frac{\tilde \gamma_t'}{|\tilde \gamma_t'|}} \Big{)}^2(\vec{\kappa}_{\tilde \gamma_t})
+ |\vec{\kappa}_{\tilde \gamma_t}|^2 \vec{\kappa}_{\tilde \gamma_t} + \vec{\kappa}_{\tilde \gamma_t} \Big{)}     
\equiv - \,\nabla_{L^2}\Wil(\tilde \gamma_t),
\end{eqnarray}
starting in $\gamma_0$, converges fully and smoothly 
to a \emph{geodesic} elastica in $\sphere^2$ 
which covers some great circle in $\sphere^2$ exactly twice.    
\end{theorem}  
\proof In order to prove Theorem \ref{bad.convergence.2}   
we will essentially follow the ideas and arguments of 
the proof of Theorem 6.1 in the author's paper \cite{Ruben.MIWF.2025}. However, the basic technical 
trick will be here to employ rather the \emph{stability} 
than \emph{instability} of a particular 
\emph{geodesic elastic curve} on the $2$-sphere.    
As in that proof we firstly have to guess the initial path 
of a suitable flow line $\{\gamma_t\}_{t \geq 0}$   
of the elastic energy flow (\ref{elastic.energy.flow}), 
and the initial path we choose here will be  
a particular tiny perturbation of the $2$-fold repetition  
$E \oplus E$ of the standard, constant speed-parametrization ``$E$'' of some fixed great circle on $\sphere^2$.
\footnote{We remark here that there are further alternative
choices of the initial path of some promising flow line 
of flow (\ref{elastic.energy.flow}). 
Still the idea to concretely focus on some tiny perturbation 
of the $2$-fold great circle on $\sphere^2$ appears 
to be most suitable for our particular purpose.}  
As we will see between formula (\ref{Stable.elastic.curves.2})
and inequality (\ref{distorted.curve})
we will actually have a lot of freedom  
in order to concretely define a section of the 
normal bundle along $E \oplus E$, which will perturb the 
double loop parametrization $E \oplus E$ in such a way 
that the elastic energy $\Wil$ can only \emph{increase},
without having to cope with any technical computations.  \\
As in the proof of Theorem 6.1 in \cite{Ruben.MIWF.2025} we firstly recall the important fact that the shapes of \emph{non-geodesic} closed elastic curves on $\sphere^2$ are determined by the symmetries of the Jacobi elliptic  
$\textnormal{cn}$-functions and have therefore been classified - up to rotations and reflections of $\sphere^2$ - in the first part of Proposition 6 of \cite{Ruben.Hopf.Willmore.2023}.
On account of this classification every \emph{non-geodesic} elastic curve on $\sphere^2$ can be uniquely characterized by its number $n \in \nat$ of consequtive lobes and 
its number $m\in \nat$ of trips along some appropriately chosen ``equator'' before it closes up, 
where $m$ and $n$ have to be coprime positive integers 
satisfying $\frac{m}{n}\in (0,2-\sqrt{2})$; see here 
formula (104) in \cite{Ruben.Hopf.Willmore.2023}. 
We therefore adopt the notation ``$\gamma_{(m,n)}$'' from Proposition 6 in \cite{Ruben.Hopf.Willmore.2023} also here, where $\gamma_{(m,n)}$ represents the isometry class of all elastica on $\sphere^2$ having $n$ consequtive lobes while they perform exactly $m$ trips along some fixed great 
circle on $\sphere^2$. \\
Next, concerning the \emph{second variation} of the elastic energy $\Wil$ in its critical points we recall here Langer's and Singer's Theorem 3.1 in their paper \cite{Langer.Singer.1987}, which implies particularly that each \emph{non-geodesic} elastic curve on $\sphere^2$
- performing only one loop through its trace -  
and additionally the $f$-fold cover of any great circle, 
for each $f\geq 4$, is an $\Wil$-\emph{unstable} critical point of the elastic energy, whereas 
\emph{the single, $2$-fold and $3$-fold cover of any great circle} is an $\Wil$-\emph{stable} critical point of the elastic energy. Focussing on the $2$-fold cover of some great circle in $\sphere^2$, the latter statement means 
precisely that the second variation of 
$\Wil$ in the constant speed-parametrization 
$\gamma^*:=E\oplus E$ of this double loop and in direction 
of any smooth normal vector field 
$\vec F_{\gamma^*}:=\phi_1\, N_{E} \oplus \phi_2\, N_{E}$
along $\gamma^*$ satisfies:
\begin{equation}  \label{Stable.elastic.curves}
(\delta^2\Wil)_{\gamma^*}
(\vec F_{\gamma^*},\vec F_{\gamma^*})
= 8 \pi\, \int_0^1 \Big{(}\frac{\tilde z''(t)}{4\pi}  
+ 4\pi  \tilde z(t) \Big{)}\, \Big{(}\frac{\tilde z''(t)}{4\pi} 
+ 2\pi \tilde z(t) \Big{)}\, dt  \geq 0.
\end{equation}  
Here $N_{E}$ denotes one of the two unit normal vector fields along the fixed great circle $E$, 
$\phi_1,\phi_2: \sphere^1 \setminus \{1\} \longrightarrow \rel$ denote two arbitrarily chosen, but compatible smooth functions such that $\vec F_{\gamma^*}$
is a well-defined and smooth normal vector field along 
the double loop $\gamma^*$, and 
$\tilde z:[0,1]\longrightarrow \rel$ denotes the primitive with $\tilde z(0)=0$ of the periodic, rescaled function 
$t \mapsto (\phi_1\, \oplus \phi_2)(\exp(4\pi it))$ on 
$\rel/\ganz$; 
compare here with p. 147 in \cite{Langer.Singer.1987}.  \\ 
Only the fact that the weak inequality in (\ref{Stable.elastic.curves}) holds 
for any smooth normal variation along $\gamma^*$, i.e. 
for any variational vector field $\vec F_{\gamma^*}$ 
along $\gamma^*$ of the mentioned form 
$\phi_1\, N_{E} \oplus \phi_2\, N_{E}$, does not automatically imply that the double loop $E\oplus E$ is a 
\emph{local minimizer} - say with respect to the $C^2(\sphere^1,\rel^3)$-norm - 
of the functional $\Wil$ in 
$C^{\infty}_{\textnormal{reg}}(\sphere^1,\sphere^2)$, 
and in view of our modest goal it would be unnecessarily difficult to clarify here, whether $E\oplus E$ is actually 
a local minimizer of $\Wil$ in 
$C^{\infty}_{\textnormal{reg}}(\sphere^1,\sphere^2)$. 
However, one can use the simple, concrete structure of the path $\gamma^*=E\oplus E$ respectively of the associated 
vector space of smooth normal vector fields 
$\phi_1\, N_{E} \oplus \phi_2 N_{E}$
from (\ref{Stable.elastic.curves}), in order to derive
the useful formula 
\begin{eqnarray}  \label{Stable.elastic.curves.2}
(\delta^2\Wil)_{E\oplus E}
(\phi_1\, N_{E} \oplus \phi_2 N_{E},
\phi_1\, N_{E} \oplus \phi_2 N_{E})    \nonumber    \\
= (\delta^2\Wil)_{E}(\phi_1\, N_{E},\phi_1 N_{E})
+(\delta^2\Wil)_{E}(\phi_2\, N_{E},\phi_2 N_{E}) 
\end{eqnarray}  
for those pairs of smooth functions 
$\phi_1,\phi_2:\sphere^1 \setminus \{1\} \longrightarrow \rel$ which are actually smooth on the entire circle $\sphere^1$, i.e. exactly for those pairs of smooth functions 
$\phi_1,\phi_2$ on $\sphere^1 \setminus \{1\}$ 
whose values close up smoothly at the 
particular point $1$ in $\sphere^1$ respectively, because 
in this case the integral on the right-hand side in (\ref{Stable.elastic.curves}) splits exactly at $t=1/2$ 
into two seperate integrals which compute the 
second variations of $\Wil$ in the single loop 
$E$ and in directions of the smooth normal vector fields 
$\phi_j\, N_{E}$, $j=1,2$, at a time.  
On the other hand, it is well-known that a single loop 
$E$ through a fixed great circle is the unique global minimizer of the elastic energy $\Wil$ in entire $C^{\infty}_{\textnormal{reg}}(\sphere^1,\sphere^2)$ - 
up to smooth reparametrizations and isometries of 
$\sphere^2$. Now we can consider for instance the prominent 
function $f(x)=\exp(-\frac{1}{x^2})$, for 
$x\in \rel\setminus \{0\}$, whose extension by 
$f(0):=0$ into the origin is of class $C^{\infty}(\rel)$ 
and satisfies $f^{(j)}(0)=0$ for every $j\in \nat_0$. This function can be used to construct a $2\pi$-periodic 
$C^{\infty}$-function $\Phi$ on $\rel$ 
satisfying $\Phi^{(j)}(0)=0=\Phi^{(j)}(2\pi)$ 
for every $j\in \nat_0$ and additionally $\Phi(t)>0$ 
for every $t \in (0,2\pi)$. Hence, introducing two 
functions $\phi_1,\phi_2: \sphere^1 \longrightarrow \rel$ by        
$$ 
\phi_1(\exp(it)):=\Phi(t) \quad \textnormal{and} \quad   
\phi_2(\exp(it)):=-\Phi(t), \quad \textnormal{for} \, 
t\in [0,2\pi],
$$ 
we can easily check that they are well-defined and 
$C^{\infty}$-smooth on entire $\sphere^1$, that there holds $\phi_1(1)=0=\phi_2(1)$ 
and that the values of their derivatives of any order  
coincide at the point $1\in \sphere^1$ - such that the resulting normal vector field  
$\vec F_{\gamma^*}:=\phi_1\, N_{E} \oplus \phi_2\, N_{E}$
along $\gamma^*=E \oplus E$ is well-defined and smooth - and that they will certainly generate \emph{two positive values} in the sum on the right-hand side of (\ref{Stable.elastic.curves.2}).
Furthermore we can immediately infer from  
the special construction of the two 
functions $\phi_1,\phi_2$ on $\sphere^1$ 
respectively of the normal vector field
$\vec F_{\gamma^*}=\phi_1\, N_{E} \oplus \phi_2\, N_{E}$
that the corresponding perturbation 
$\exp_{\gamma^*}(\varepsilon \vec F_{\gamma^*})$ 
of the double loop $\gamma^*=E\oplus E$ is a 
$C^{\infty}$-smooth, closed, and regular path on 
$\sphere^2$ having only one point of self-intersection, 
at least for sufficiently small and positive $\varepsilon$. \\
Now, inserting this special result regarding the left-hand side in (\ref{Stable.elastic.curves.2}) into a Taylor-expansion 
of the $C^{\infty}$-smooth function $\varepsilon \mapsto 
\Wil(\exp_{\gamma^*}(\varepsilon \vec F_{\gamma^*}))$
about $\varepsilon=0$ we conclude, that the 
energy function $\varepsilon \mapsto 
\Wil(\exp_{\gamma^*}(\varepsilon \vec F_{\gamma^*}))$ 
has to increase above $\Wil(\gamma^*)$, for 
$\varepsilon \in (0,\varepsilon_0)$, provided 
$\varepsilon_0>0$ is chosen here sufficiently small.   
Hence, we have determined a certain family  
$\{\gamma^{\varepsilon}\}_{\varepsilon \in [0,\varepsilon_0)}
:=\{\exp_{\gamma^*}(\varepsilon \vec F_{\gamma^*})\}
_{\varepsilon \in [0,\varepsilon_0)}$ 
of $C^{\infty}$-smooth and regular closed paths, 
with $\gamma^0 = E\oplus E \equiv \gamma^*$, 
such that for every fixed $\varepsilon \in (0,\varepsilon_0)$ 
the path $\gamma^{\varepsilon}$ travels $2$ times about $\sphere^2$ along the chosen great circle $\textnormal{trace}(E)$ - just as $\gamma^*$ itself does -
\emph{intersects itself only once} and satisfies: 
\begin{equation} \label{distorted.curve} 
13 > \Wil(\gamma^{\varepsilon}) 
>\Wil(\gamma^*) = 4 \pi,
\end{equation}  
provided $\varepsilon_0>0$ was chosen sufficiently small.
On account of inequality (\ref{distorted.curve}) 
we shall henceforth consider the unique flow line 
$\{\gamma^{\varepsilon}_t\}$ of flow 
(\ref{elastic.energy.flow}) starting in such 
a distorted closed path $\gamma^{\varepsilon}$, for some 
arbitrarily chosen $\varepsilon \in (0,\varepsilon_0)$,    
and we can start to prepare our final argument - 
a topological argument which is simpler than the one 
that finished the proof of 
Theorem 6.1 in \cite{Ruben.MIWF.2025} and is 
\emph{not an argument by contradiction}. To this end, 
we recall from the proof of Theorem 1, parts (I) and (II), 
in \cite{Ruben.Hopf.Willmore.2023} respectively from the main result of \cite{Dall.Acqua.Pozzi.2018} - here additionally combined with Theorem 1.2 in \cite{Dall.Acqua.Spener.2016} - that $\{\gamma^{\varepsilon}_t\}_{t \geq 0}$ exists globally   
and converges fully and smoothly to a particular elastic curve 
$\gamma^{\varepsilon}_{\infty}$ in $\sphere^2$. 
More precisely, there exists some smooth family of smooth diffeomorphisms 
$\varphi_t:\sphere^1 \stackrel{\cong}\longrightarrow \sphere^1$ such that
\begin{equation} \label{full.smooth.convergence}
\gamma^{\varepsilon}_t \circ \varphi_t 
\longrightarrow \gamma^{\varepsilon}_{\infty} \,\,\, \textnormal{in}\,\,\, C^{k}(\sphere^1,\rel^3), \,\,\, 
\textnormal{for every}\,\,k \in \nat_0, 
\end{equation}	
as $t\to \infty$, where the limit path 
$\gamma^{\varepsilon}_{\infty}:\sphere^1 \longrightarrow \sphere^2$ parametrizes a smooth elastic curve on $\sphere^2$ with constant speed - possibly performing several loops through its trace. Since $\gamma^{\varepsilon}_0\equiv \gamma^{\epsilon}$ satisfies inequality (\ref{distorted.curve}), we can infer from the monotonicity 
of the elastic energy along the global flow line $\{\gamma^{\varepsilon}_t\}$ that  
\begin{equation}  \label{Wil.goes.down}
\Wil(\gamma^{\varepsilon}_{\infty})
\leq \Wil(\gamma^{\varepsilon}) < 13.
\end{equation} 
Following the strategy of the proof of Theorem  
6.1 in \cite{Ruben.MIWF.2025} we are going to use  
estimate (\ref{Wil.goes.down}) in order to verify, 
that \emph{all non-geodesic elastic curves} $\gamma_{(m,n)}$ 
must have higher elastic energies than both our chosen 
initial path $\gamma^{\varepsilon}$ and the corresponding  
limit path $\gamma^{\varepsilon}_{\infty}$ from 
(\ref{full.smooth.convergence}) have. 
To this end, we quickly recall here from Section 4 in 
\cite{Langer.Singer.1987} and from the proofs 
of Proposition 6 in \cite{Ruben.Hopf.Willmore.2023} 
and Theorem 6.1 in \cite{Ruben.MIWF.2025} that the wavelength 
$\Lambda(\gamma_{(m,n)})=\frac{m}{n}2 \pi$ of the 
closed non-geodesic elastica ``$\gamma_{(m,n)}$'' 
can be systematically recovered as certain values of the function (102) in \cite{Ruben.Hopf.Willmore.2023}, which depends on the modulus $p \in \big{(}0,\frac{1}{\sqrt{2}}\big{)}$ 
of the Jacobi elliptic function appearing in formula (97) 
of \cite{Ruben.Hopf.Willmore.2023},   
and this particular function turned out to be 
\emph{strictly monotonically decreasing} on account of 
formula (103) in \cite{Ruben.Hopf.Willmore.2023}. 
This insight has two important ramifications: 
($\alpha$) it yields a one-to-one correspondence between 
all quotients $\frac{m}{n}\in (0,2-\sqrt{2})$ with $\textnormal{gcd}(m,n)=1$ and exactly those moduli 
$p=p(m,n)\in \big{(}0,\frac{1}{\sqrt{2}}\big{)}$ which 
produce the squared curvatures of 
\emph{non-geodesic closed elastic curves} on $\sphere^2$
on account of formula (97) in \cite{Ruben.Hopf.Willmore.2023} -
leading in turn to the above mentioned statement 
of the first part of Proposition 6 in \cite{Ruben.Hopf.Willmore.2023} - and   
($\beta$) it hints at a possibility to explicitly compute 
the energy $\Wil$ of any non-geodesic elastic curve $\gamma_{(m,n)}$ in terms of certain elliptic integrals which depend on the uniquely corresponding modulus $p=p(m,n)$, 
and this aim was finally accomplished in formula (106) of \cite{Ruben.Hopf.Willmore.2023}: 
\begin{equation} \label{energy.directly}    
\Wil(\gamma_{(m,n)}) = \frac{8n}{\sqrt{2-4(p(m,n))^2}} \,\big{(}2E(p(m,n))-K(p(m,n))\big{)},  
\end{equation} 	             
compare here also with formula (159) in \cite{Ruben.MIWF.2025}.  
Moreover, of fundamental importance for the proof 
of estimate (\ref{Wil.estimate.below}) below is the additional 
observation that the function 
$f(p):=\frac{1}{\sqrt{1-2p^2}} \,(2E(p)-K(p))$, 
appearing on the right-hand side of formula (\ref{energy.directly}), is strictly monotonically increasing on its entire domain $\big{(}0,\frac{1}{\sqrt{2}}\big{)}$; 
see here the computations on p. 38 in \cite{Ruben.Hopf.Willmore.2023}. \\   
Now, as in formula (161) of \cite{Ruben.MIWF.2025} 
we fix some number $m^*$ of trips about some  
great circle in $\sphere^2$ - here in this proof  
without any additional condition on $m^*\in \nat$ - 
and we roughly estimate the elastic energy $\Wil(\gamma_{(m^*,n)})$ from below by means 
of formula (\ref{energy.directly}), 
combined with the above mentioned monotonicity of the function $f(p)=\frac{1}{\sqrt{1-2p^2}} \,(2E(p)-K(p))$ 
on $\big{(}0,\frac{1}{\sqrt{2}}\big{)}$ and the particular 
facts that $f(0)= 2E(0)-K(0)=\frac{\pi}{2}$ 
and that each \emph{admissible pair} of indices $(m,n)$ satisfies\, $0<\frac{m}{n}<2-\sqrt{2}$:
\begin{equation} \label{Wil.estimate.below}     
\Wil(\gamma_{(m^*,n)}) 
> \frac{8\cdot \frac{3}{2}\,m^*}{\sqrt{2}} \,
f(p(m^*,n)) \geq \frac{6\,m^*}{\sqrt{2}} \,\pi 
\geq \frac{6}{\sqrt{2}} \,\pi \approx 13.32865, 
\end{equation} 
for every admissible pair $(m^*,n)$, i.e. for 
every trip number $m^*\in \nat$.  
Obviously, the lower bound $13.32865$ is still larger than 
$\Wil(\gamma^{\varepsilon}_{\infty})$ by estimate  (\ref{Wil.goes.down}). Hence 
comparing estimate (\ref{Wil.estimate.below})
with statement (\ref{Wil.goes.down}) and with the energies 
of geodesic elastica on $\sphere^2$, i.e. with 
the values $2k\pi$ for each $k\in \nat$, 
we can easily conclude that there are only the following 
two different possibilities for the 
\emph{elastic} limit path in (\ref{full.smooth.convergence}): 
\begin{equation} \label{two.terminal.curves}
(A):\,\, \gamma^{\varepsilon}_{\infty} = E \quad \,\,\, 
(B): \,\, \gamma^{\varepsilon}_{\infty}=E \oplus E,
\end{equation}  
at least up to appropriate isometries of $\sphere^2$.
Similarly, we can infer from a comparison of both inequalities 
in (\ref{distorted.curve}) with the energies $2k\pi$
of geodesic elastica on $\sphere^2$
and again with estimate (\ref{Wil.estimate.below}) 
for every trip number $m^*\in \nat$, 
that our favoured candidate 
$\{\gamma^{\varepsilon}_t\}_{t \geq 0}$ cannot start 
moving in an elastic curve and hence has to be  
a non-trivial solution of (\ref{elastic.energy.flow}).
In other words, the elastic energy 
has to decrease strictly monotonically along 
$\{\gamma^{\varepsilon}_t\}$ respectively along its 
smoothly convergent reparametrization 
$\{\gamma^{\varepsilon}_t \circ \varphi_t\}$ 
from (\ref{full.smooth.convergence}). \\ 
As in the proof of Theorem 6.1 in \cite{Ruben.MIWF.2025} 
we are finally going to combine the above insights with a particular, appropriate \emph{homotopy invariant} on the set  $C^{\infty}_{\textnormal{reg}}(\sphere^1,\sphere^2)$ 
of smooth closed regular paths on $\sphere^2$. 
\footnote{In the proof of Theorem 6.1 in \cite{Ruben.MIWF.2025} we were particularly interested 
in more sophisticated invariants on the subset of \emph{generic} smooth closed paths in $\rel^2$ respectively $\sphere^2$ in the sense of Arnold \cite{Arnold.1993}, which is not necessary in our relatively simple situation.}   
We cannot a-priori rule out the possibility that the considered global flow line $\{\gamma^{\varepsilon}_t\}_{t\geq 0}$ might cover the entire $2$-sphere. In this case, there is no point $b$ 
on the $2$-sphere such that some appropriate stereographic projection $\Psi_b:\sphere^2 \setminus \{b\} \longrightarrow \rel^2$ would project the entire flow line $\{\gamma^{\varepsilon}_t\}_{t\geq 0}$ from $\sphere^2$ into $\rel^2$. 
Therefore, we cannot apply here \emph{Whitney-Graustein's} 
tangent-rotation number ``$\textnormal{ind}$'' 
on $C^{\infty}_{\textnormal{reg}}(\sphere^1,\rel^2)$ - 
a fairly elementary invariant which reflects the basic topological fact that 
$\pi_1(\sphere^1,(1,0))\cong \ganz$. 
Instead, we are going to invoke the well-known 
diffeomorphism between the subbundle  
$T^1\sphere^2$ of the tangent bundle of $\sphere^2$ 
consisting of \emph{unit} tangent vectors and the Lie group $\textnormal{SO}(3)$ and then use the prominent fact that $\pi_1(\textnormal{SO}(3),\textnormal{Id}) 
\cong \ganz/2$. Let's construct this 
particular invariant 
$\textnormal{ind}_2:
C^{\infty}_{\textnormal{reg}}(\sphere^1,\sphere^2)
\longrightarrow \ganz/2$ concretely: \\
Every closed and \emph{regular} path 
$\gamma:\sphere^1 \longrightarrow \sphere^2$ gives 
rise to its field of unit tangents 
$\frac{\gamma'(s)}{|\gamma'(s)|}\in T^1_{\gamma(s)}\sphere^2$ along $\gamma$ which is first of all a closed path in $\rel^3$, but more precisely a closed path in the total space of the circle-bundle $T^1\sphere^2$. 
Hence, the field of unit tangents 
$s\mapsto \frac{\gamma'(s)}{|\gamma'(s)|}$ of any closed regular path $\gamma$ becomes a closed path in $\textnormal{SO}(3)$ when being composed with the standard diffeomorphism $\Theta:T^1\sphere^2\stackrel{\cong}\longrightarrow \textnormal{SO}(3)$. It is well-known 
that the topological meaning of the isomorphism $\pi_1(\textnormal{SO}(3),\textnormal{Id}) 
\cong \ganz/2$ is given by the fact that any non-constant, closed and continuous path in $\textnormal{SO}(3)$ is already contractible within $\textnormal{SO}(3)$ if it is homotopic 
in $\textnormal{SO}(3)$ to some closed path which consists 
of an \emph{even number} of identical loops. 
This motivates us to consider the map
$$ 
\textnormal{ind}_2:
C^{\infty}_{\textnormal{reg}}(\sphere^1,\sphere^2)
\longrightarrow \pi_1(\textnormal{SO}(3))  
\cong \ganz/2 \,\,\, \textnormal{by} \,\,\,
\gamma \mapsto \Big{[}s\mapsto \Theta 
\Big{(}\frac{\gamma'(s)}{|\gamma'(s)|}\Big{)}\Big{]}. 
$$  
Obviously, this map is an \emph{invariant} with respect to regular homotopies between smooth, closed and regular paths in $\sphere^2$, because any \emph{regular} homotopy 
$H:\sphere^1 \times [0,1] \longrightarrow \sphere^2$ 
between two paths $\gamma_1, \gamma_2 \in 
C^{\infty}_{\textnormal{reg}}(\sphere^1,\sphere^2)$ 
yields the well-defined and continuous map 
$(s,t) \mapsto \frac{\partial_s H(s,t)}{|\partial_s H(s,t)|}$ 
from $\sphere^1 \times [0,1]$ into $T^1\sphere^2$,       
which constitutes a homotopy 
between the fields of unit tangents of $\gamma_1$ and 
$\gamma_2$ in $T^1\sphere^2$, implying 
$\textnormal{ind}_2(\gamma_1)=\textnormal{ind}_2(\gamma_2)$, 
as required.  
Moreover we infer from the above discussion that the index $\textnormal{ind}_2$ of some path $\gamma \in C^{\infty}_{\textnormal{reg}}(\sphere^1,\sphere^2)$ 
vanishes, if and only if the induced closed path  
$s\mapsto \Theta \Big{(}\frac{\gamma'(s)}{|\gamma'(s)|}\Big{)}$ in $\textnormal{SO}(3)$ is homotopic to some closed path in 
$\textnormal{SO}(3)$ which traverses an \emph{even number} of 
identical loops. \\
Now, the considered closed initial path   
$\gamma^{\varepsilon}$ was constructed in such a way that 
it can be deformed into the standard parametrization 
of a $2$-fold covered great circle in $\sphere^2$ 
by means of an appropriate regular homotopy. 
Hence, by invariance of our $\ganz/2$-index with 
respect to regular homotopies we immediately conclude that $\textnormal{ind}_2(\gamma^{\varepsilon})
=\textnormal{ind}_2(E\oplus E)=0$. 
On the other hand, we also obtain the values of our $\ganz/2$-index on the two possible elastic limit 
curves in (\ref{two.terminal.curves}) 
without any computation: the simple loop $E$ in option (A) 
has $\ganz/2$-index $1$, whereas   
the double loop in option (B) has $\ganz/2$-index $0$.
Since the restriction 
$\{\gamma^{\varepsilon}_t \circ \varphi_t\}_{t\in [0,T]}$
of the reparametrized convergent flow line $\{\gamma^{\varepsilon}_t \circ \varphi_t\}_{t \geq 0}$ 
in (\ref{full.smooth.convergence}) to some arbitrarily 
large compact interval $[0,T]$ is actually a 
\emph{regular homotopy} between smooth, closed and regular paths on $\sphere^2$, we only have to compare the computed 
values of $\textnormal{ind}_2(\gamma^{\varepsilon}_{\infty})$ 
on the two remaining possible limit curves in (\ref{two.terminal.curves}) with the value $\textnormal{ind}_2(\gamma^{\varepsilon})=0$ 
and conclude from the homotopy invariance of 
$\textnormal{ind}_2$ that the limit in (\ref{full.smooth.convergence}) of 
$\{\gamma^{\varepsilon}_t \circ \varphi_t\}$   
must be the constant speed-parametrization of some 
$2$-fold covered great circle on $\sphere^2$, i.e. that
\begin{equation} \label{full.smooth.convergence.2}
\gamma^{\varepsilon}_t \circ \varphi_t 
\longrightarrow E \oplus E \,\,\, \textnormal{in}\,\,\, 
C^{k}(\sphere^1,\rel^3), \,\,\, 
\textnormal{for every}\,\,k \in \nat_0, 
\end{equation}	
as $t\to \infty$. Finally, because of  
$\Wil(E \oplus E)=4\pi$ and since the elastic energy 
depends continuously on the perturbation parameter 
$\varepsilon$ of the smooth family  
$\gamma^{\varepsilon}
=\exp_{\gamma^*}(\varepsilon \vec F_{\gamma^*})$
for $\varepsilon \in [0,\varepsilon_0)$, perturbing 
exactly the path $E\oplus E \equiv \gamma^*$,  
we can thus choose for every $\delta>0$ some 
sufficiently small $\varepsilon$ in $(0,\varepsilon_0)$ 
such that 
$4\pi <\Wil(\gamma^{\varepsilon})<4\pi+\delta$ holds.
Moreover, as pointed out above, for every $\varepsilon$ in $(0,\varepsilon_0)$ the closed path 
$\gamma^{\varepsilon}$ intersects itself only once.   
Hence, on account of the convergence of 
the reparametrized flow line
$\{\gamma^{\varepsilon}_t \circ \varphi_t\}$ in  (\ref{full.smooth.convergence.2}) the path 
$\gamma^{\varepsilon}$ has indeed all those properties 
which have been required from the initial path ``$\gamma_0$'' 
in the formulation of Theorem \ref{bad.convergence.2}, and Theorem \ref{bad.convergence.2} is proved.     
\qed

\end{document}